\documentclass[letter]{amsart}
\usepackage{amsmath,amssymb,amsthm}
\usepackage{indentfirst}
\usepackage{hyperref}
\usepackage{cleveref}

\newtheorem{prop}{Proposition}
\newtheorem{lemma}[prop]{Lemma}
\newtheorem{theorem}[prop]{Theorem}

\newtheorem*{mainThm}{Theorem 1}

\theoremstyle{definition}
\newtheorem*{mydef}{Definition}
\newtheorem{remark}[prop]{Remark}
\newtheorem{example}[prop]{Example}
\newtheorem*{question}{Question}

\everymath{\displaystyle}

\DeclareMathOperator{\Ass}{Ass}
\DeclareMathOperator{\Hom}{Hom}
\DeclareMathOperator{\Ext}{Ext}

\DeclareMathOperator{\rank}{rank}
\DeclareMathOperator{\height}{height}
\DeclareMathOperator{\type}{type}

\newcommand{\p}{\mathfrak{p}}
\newcommand{\q}{\mathfrak{q}}
\newcommand{\m}{\mathfrak{m}}
\newcommand{\gE}{{^{*} \! E}}

\begin{document}
 \title{\bf{Graded-irreducible modules are irreducible}}

\author{Justin Chen}
\address{Department of Mathematics, University of California, Berkeley,
California, 94720 U.S.A}
\email{jchen@math.berkeley.edu}

\author{Youngsu Kim}
\address{Department of Mathematics, University of California, Riverside,
California, 92521 U.S.A}
\email{youngsu.kim@ucr.edu}

 %; Secondary {} } 

% \date
% \today

\begin{abstract}
We show that if a graded submodule of a Noetherian module
cannot be written as a proper intersection of graded submodules, then it 
cannot be written as a proper intersection of submodules at all. More generally,
we show that a natural extension of the index of reducibility to the graded setting
coincides with the ordinary index of reducibility. We also investigate the question
of uniqueness of the components in a graded-irreducible 
decomposition, as well as the relation between the index of reducibility of a 
non-graded ideal and that of its largest graded subideal.
\end{abstract}

\maketitle

\medskip

Let $R$ be a ring (commutative with $1 \neq 0$),
$M$ an $R$-module, and $N \subseteq M$ a submodule. 
Recall that $N$ is said to be \textit{irreducible} in $M$ if
whenever $N = N_1 \cap N_2$ for some $R$-submodules $N_1, N_2$ of $M$, then 
$N = N_1$ or $N = N_2$.

\begin{mydef} Let $R$ be a $\mathbb{Z}$-graded ring, $M$ a graded $R$-module,
and $N \subseteq M$ a graded submodule. $N$ is said to be 
\textit{graded-irreducible} in $M$ if whenever $N = N_1 \cap N_2$ for some graded 
$R$-submodules $N_1, N_2$ of $M$, then $N = N_1$ or $N = N_2$.
\end{mydef}

Equivalently, $N$ is graded-irreducible iff whenever $N = N_1 \cap \ldots \cap 
N_r$ for a finite collection of graded submodules $N_i$, then $N = N_i$ for some $i$. \\

It follows directly from the definitions that if $N$ is a graded submodule of $M$ and
$N$ is irreducible, then $N$ is graded-irreducible. A natural question to ask is whether
or not the converse holds. This is answered by our first result: \\

\begin{theorem}\label{mainTheorem}
Let $R$ be a $\mathbb{Z}$-graded ring, $M$ a Noetherian graded $R$-module, 
and $N \subseteq M$ a graded submodule. Then $N$ is irreducible iff $N$ is 
graded-irreducible.
\end{theorem}

As an analogy, we recall the case of monomial (i.e. $\mathbb{N}^n$-graded) ideals 
in a polynomial ring. It is a textbook exercise that monomial ideals are irreducible 
iff they are monomial-irreducible (i.e. cannot be written as a proper intersection of 
monomial ideals) iff they are generated by pure powers of the variables, cf. 
\cite[Section 1.3.1]{HH}. Moreover, decompositions into monomial-irreducible 
ideals are unique. \\

We note that it is possible for even a monomial ideal to be an intersection of 
non-graded ideals, e.g. $(x^2,xy,y^3) = (x^2,xy,x-y^2) \cap (x^2,xy,x+y^2)$
in $k[x,y]$, for $k$ a field, $\operatorname{char} k \ne 2$.
However, the monomial case is quickly resolved by the following lemma 
(cf. \cite{MS}, Lemma 5.18): If $I$ is a monomial ideal with a minimal generator 
$mm'$ where $m, m'$ are relatively prime monomials, then $I = (I + (m)) \cap (I + (m'))$. 
In contrast, there is no such formula in the $\mathbb{Z}$-graded case; not to 
mention that a general graded ring can be far worse behaved than a polynomial ring. \\

Before giving the proof of \Cref{mainTheorem}, we introduce a numerical invariant of a 
graded submodule. If $M$ is Noetherian, every graded submodule of $M$ is a finite 
intersection of graded-irreducible submodules: If there were a maximal
counterexample $N$, then $N$ would not be graded-irreducible. By definition
$N$ would be an intersection of two strictly larger graded submodules, which (by 
maximality of $N$) are finite intersections of graded-irreducibles, and thus
$N$ is as well, a contradiction. This motivates the following definition: 

\begin{mydef} Let $M$ be a Noetherian $R$-module and $N \subseteq M$ a submodule.
The \textit{index of reducibility} of $N$ in $M$ is 
$$r_M(N) := \min \left\{ r : \exists N_1, \ldots, N_r \text{ irreducible $R$-submodules, } 
N = \bigcap_{i=1}^r N_i \right\}.$$
If $R, M, N$ are graded, the \textit{graded index of reducibility} of $N$ in $M$ is 
$$r^g_M(N) := \min \left\{ r : \exists N_1, \ldots, N_r \text{ graded-irreducible 
$R$-submodules, } N = \bigcap_{i=1}^r N_i  \right\}.$$
If the module $M$ is understood, e.g. if $N = I$ is an $R$-ideal (so that $M = R$), then
we may simply write $r(N)$ or $r(I)$. The Noetherian hypothesis on $M$ 
guarantees that $r_M(N)$, $r^g_M(N)$ are both always finite. Moreover, $r_M(N) = 1$ 
iff $N$ is irreducible, and if $N$ is graded, $r^g_M(N) = 1$ iff $N$ is graded-irreducible.
\end{mydef}

When $M = R$ is local and $N$ is primary to the maximal ideal, 
the definition of the index of reducibility above has appeared in the literature 
(although to the best of our knowledge, the graded index of reducibility 
has not appeared before), and is well-known to be the vector space dimension 
of the socle over the residue field of $R$, see e.g. \cite{Grobner, No, GS}. 
If in addition $N$ is graded and $R$ is a local $\mathbb{N}$-graded ring, 
then it is not hard to show that both indices coincide, cf. \Cref{typeLemma}. 
However, in the non-local case, e.g. if $R$ is a polynomial ring, the graded 
index of reducibility is in general not a vector space dimension over the base field; 
one must compute ranks over a Laurent polynomial ring instead.\\

From the definitions alone, it is not clear if there is any relation between $r_M(N)$ 
and $r^g_M(N)$ that holds in general. Somewhat surprisingly, they are always equal; 
a fact which follows from \Cref{mainTheorem}. Indeed, the following three statements 
are equivalent for a Noetherian graded module $M$, cf. 
\Cref{equivTheorem}:

\begin{enumerate}
\item Any graded-irreducible submodule is irreducible.
\item For any graded submodule $N$, $r_M (N) = r_M^g (N)$. 
\item Any graded submodule is a finite intersection of irreducible graded submodules.
\end{enumerate}

The key feature of Statement $(3)$ is the simultaneous requirements 
of finiteness, irreducibility, and gradedness (notice: finiteness along
with either irreducibility or gradedness is easy to satisfy). 
Although an independent proof of Statement $(3)$, which at first sight may 
seem to follow directly from the Noetherian hypothesis, would give 
another proof of \Cref{mainTheorem}, so far we have been unable to find one.

\section*{Proof of Theorem 1}
We begin the proof of Theorem 1 with some reductions, which we use hereafter
without further mention. From the definitions, $r_M(N) = r_{M/N}(0)$, and if $N$ is graded,
$r^g_M(N) = r^g_{M/N}(0)$. 
Also, $R$-submodules of $M$ are the same as $R/\!\operatorname{ann}_R M$-submodules 
of $M$ (and if $R$, $M$ are graded, then $R/\operatorname{ann}_R M$ is also graded), so henceforth we will pass to the factor ring $R/\!\operatorname{ann}_R M$ and 
assume that the ring $R$ is Noetherian. If $\p$ is a graded prime ideal of a graded ring $R$, for any graded $R$-module $M$ we set $M_{(\p)} := W^{-1}M$, where $W$ is the set of all homogeneous elements of $R$ not in $\p$, so $M_{(\p)}$ is a graded $R_{(\p)}$-module. Finally, we refer to Sections $1.5$ and $3.6$ in 
\cite{BH} for notation and basic results for the graded case.

\begin{lemma}\label{localizationLemma}
Let $R$ be a Noetherian ring, $\p$ a prime ideal of $R$, $M$ a finitely generated $R$-module, and $N \subseteq M$
a submodule. If $N$ is $\p$-primary, then $N$ is irreducible in $M$ iff $N_\p$ is irreducible 
in $M_\p$. If in addition $R, M, N$, and $\p$ are graded, then $N$ is graded-irreducible in $M$ iff 
$N_{(\p)}$ is graded-irreducible in $M_{(\p)}$.

\end{lemma}

\begin{proof}
We may assume $N = 0$. Then $\Ass(M) = \{\p\}$, so $R \setminus \p$ consists of 
non-zerodivisors on $M$, and thus the localization map $i : M \to M_\p$ is injective. If 
$0 = N_1 \cap N_2$ for $0 \ne N_1, N_2 \subseteq M$, then $(N_1)_\p \cap (N_2)_\p 
= (N_1 \cap N_2)_\p = 0$ with $(N_1)_\p, (N_2)_\p \ne 0$. 
Conversely, 
submodules of $M_\p$ are extended from submodules of $M$, so if $0 = (N_1)_\p 
\cap (N_2)_\p$ for $0 \ne N_1, N_2 \subseteq M$, then 
$0 = i^{-1}(0) = i^{-1}((N_1)_\p) \cap i^{-1}((N_2)_\p) \supseteq N_1 \cap N_2$. 

In the graded case, the same proof above applies to the homogeneous localization
$M \hookrightarrow M_{(\p)}$.
\end{proof}

\begin{remark}
Despite the elementary nature of the proof of \Cref{localizationLemma}, the conditions 
are quite delicate: In general, irreducibility is not preserved under faithfully flat ring 
extensions or localizations. For an example where irreducibility in the source does 
not extend to the target, take $k \hookrightarrow k \times k$ for a field $k$ or the 
completion $k[x,y]_{(x,y)}/(y^2 - x^3 - x^2) \hookrightarrow k[[x,y]]/(y^2 - x^3 - x^2)$. 
On the other hand, if $R$ is a domain and $\p, \q \in \operatorname{Spec}(R)$ 
with $I := \p \cap \q \ne \p, \q$, then $r_R(I) = 2$ and $R \hookrightarrow R_\p$ 
is flat and injective, but $r_{R_\p}(I_\p) = 1 = r_{R_\q}(I_\q)$. 
\end{remark}

We next give a formula for $r_M(N)$ (resp. $r^g_M(N)$) in the Artinian case in terms 
of socle rank (resp. graded socle rank). We provide a proof for the graded case; 
the proof holds verbatim in the local case, after removing all appearances of the 
word ``graded". 

\begin{lemma}\label{typeLemma}
Let $R$ be a Noetherian ring, $M$ a finitely generated $R$-module, and 
$N \subseteq M$ a submodule.

\begin{enumerate}

\item If $(R, \m, k)$ is local and $M/N$ is Artinian, then 
$r_M(N) = \dim_k (0 :_{M/N} \m).$\footnote{For the case where $N$ is a 
parameter ideal in a local ring, see \cite[Satz 3]{Grobner}.}

\item If $(R, \m, k)$ is $^*\!$local and $M/N$ is $^*\!$Artinian, then 
$r^g_M(N) = \rank_k (0 :_{M/N} \m)$.

\end{enumerate}
\end{lemma}

\begin{proof}
(2): Notice that $k = R/\m$ is a graded field, so graded $k$-modules are free,
i.e. are direct sums of twists of $k$. Replacing $M$ with $M/N$, we may 
assume $N = 0$. Set $r^g_M(0) = r$, so $0 = N_1 \cap \ldots \cap N_r$ with
$N_i$ graded-irreducible, and this decomposition is irredundant, i.e. 
$\cap_{j \ne i} N_j \not \subseteq N_i$ for each $i$. \\

Now for any graded module $M$, $0$ is graded-irreducible in $M$ iff $\gE_R(M)$ 
is $^*$indecomposable iff $\gE_R(M) = \gE_R(L)$ for every graded submodule 
$0 \ne L \subseteq M$. Thus the decomposition $0 = N_1 \cap \ldots \cap N_r$ 
implies $\gE(M) 
\cong \gE(M/N_1) \oplus \ldots \oplus \gE(M/N_r)$. Also, 
$\Hom_R(k, M) = \Hom_R(k, \gE(M))$ for any graded module $M$. \\

Since $M$ is $^*\!$Artinian, $\Ass_R(M) = \{ \m \}$. 
The structure theorem of 
$^*$injectives implies $\gE(M/N_i) \cong \gE(R/\m)$ for each $i$ (as $\gE(M/N_i)$ 
is $^*$indecomposable), so $\gE(M) \cong (\gE(k))^r.$ 
Then one has
\begin{alignat*}{2}
\rank_k (0 :_M \m) &= \rank_k \Hom_R(k, M) &&= \rank_k \Hom_R(k, \gE(M)) \\
&= \rank_k \Hom_R( k , (\gE(k))^r) &&= r \cdot \rank_k \Hom_R (k, \gE(k))\\
&= r \cdot \rank_k \Hom_R(k, k) &&= r,
\end{alignat*}
as $\Hom_R(k, k) \cong 0 :_k \m = k$.
\end{proof}

\begin{remark}\label{type}
Let $(R,\m,k)$ be a Noetherian local ring. The \textit{type} of a finitely 
generated $R$-module $M$ is defined as 
$$
\type(M) := \dim_k \Ext^{\operatorname{depth} M}_R ( k, M).
$$
Since $\Hom_R(k,M) \cong 0 :_M \m$,
\Cref{typeLemma}(1) implies that if $M$ is Artinian, then $r_M(0) = \type(M)$.
Recall that a local ring $R$ is Gorenstein iff $R$ is Cohen-Macaulay of type $1$. 
\end{remark}

\begin{remark} \label{irredCriterion}
Together, lemmas \ref{localizationLemma} and \ref{typeLemma} yield the 
following irreducibility criterion: if $I$ is an ideal in a Noetherian ring $R$, then
$I$ is irreducible iff $I$ is primary and generically Gorenstein (recall that an ideal 
$I$ is generically Gorenstein if $(R/I)_\p$ is Gorenstein for all $\p \in \Ass(R/I)$).
\end{remark}

Lemmas \ref{localizationLemma} and \ref{typeLemma}, along with one last trick, yield the theorems mentioned above:

\begin{mainThm}
Let $R$ be a $\mathbb{Z}$-graded ring, $M$ a Noetherian graded $R$-module, 
and $N \subseteq M$ a graded submodule. Then $N$ is irreducible iff $N$ is 
graded-irreducible.
\end{mainThm}
\begin{proof}
The only if direction is clear. Suppose that $N$ is graded-irreducible.
Replacing $M$ with $M/N$, we may assume that $N = 0$. 
Take a graded primary decomposition $0 = Q_1 \cap \ldots \cap Q_t$ 
with $Q_i$ graded and primary, cf. \cite[Prop. 3.12]{Ei}. By hypothesis $0 = Q_i$ for 
some $i$, i.e. $0$ is primary, so $\Ass_R(M) = \{\p\}$ for some graded
prime ideal $\p$.\\

By \Cref{localizationLemma}, 0 is graded-irreducible in $M_{(\p)}$, i.e. 
$r^g_{M_{(\p)}}(0) = 1$. As $M_{(\p)}$ is $^*\!$Artinian over the $^*$local ring 
$R_{(\p)}$, \Cref{typeLemma}(2) implies that $0 :_{M_{(\p)}} \p R_{(\p)}
\cong (R_{(\p)}/\p R_{(\p)})(n)$, for some twist $n \in \mathbb{Z}$. Localizing
at $\p$ yields $0 :_{M_\p} \p R_\p \cong R_{\p}/ \p R_{\p}$, so by \Cref{typeLemma}(1),
$r_{M_\p}(0) = 1$. By \Cref{localizationLemma}, $0$ is irreducible in $M$.
\end{proof}

We next show the equivalence mentioned in the introduction: 

\begin{theorem}\label{equivTheorem}
Let $R$ be a $\mathbb{Z}$-graded ring and $M$ a Noetherian graded
$R$-module. The following statements are equivalent:
\begin{enumerate}
\item Every graded-irreducible submodule $N \subseteq M$ is irreducible.

\item For every graded submodule $N \subseteq M$, $r_M(N) = r^g_M(N)$.

\item Every graded submodule $N \subseteq M$ is a finite intersection of irreducible graded submodules.
\end{enumerate}
\end{theorem}

\begin{proof}
$(1) \Longrightarrow (2)$: 
Take a decomposition of $N$ into graded-irreducible modules, of length $r^g_M(N)$. By (1), this is an irredundant irreducible decomposition of $N$ and by \cite[Exercise II.\S 2.17]{Bo}, every such decomposition has length $r_M(N)$. \\

$(2) \Longrightarrow (1)$: If $N$ is graded, then $N$ is irreducible iff $r_M(N) = 1$ iff $r^g_M(N) = 1$ iff $N$ 
is graded-irreducible. \\

$(1) \Longrightarrow (3)$:
Let $N \subseteq M$ be a graded submodule. 
Write $N = N_1 \cap \ldots \cap N_r$, where $N_i$ are 
graded-irreducible. By assumption, each $N_i$ is irreducible.\\

$(3) \Longrightarrow (1)$:
Let $N$ be graded-irreducible, and 
take a decomposition $N = N_1 \cap \ldots \cap N_r$
where each $N_i$ is irreducible and graded. Since $N$ is graded-irreducible, 
$N = N_i$ for some $i$, so $N$ is irreducible.
\end{proof}

\section*{Examples}

\begin{example}
Let $R$ be a Noetherian ring and $I$ a radical ideal. Then $r_R(I) = 
|\operatorname{Min}(I)| = |\Ass(I)|$, the number of minimal (equivalently associated) 
primes of $I$. This can fail if $I$ only satisfies Serre's condition $S_1$ but not $R_0$, 
as the following example shows.
\end{example}

\begin{example}\label{example} \
Let $k$ be a field, $\operatorname{char} k \ne 2$, $R := k[x,y]$, 
and set
\begin{align*}
I &:= (x^2 \; + \; xy, x^2 - y^2, y^3), \\
J_1 &:= I + (x^2 - x - y) = (x^2 - x - y, xy + x + y), \text{ and} \\
J_2 &:= I + (x^2 + x + y) = (x^2 + x + y, xy - x - y).
\end{align*}

The ideals $J_1$ and $J_2$ are $(x,y)$-primary and generated by regular sequences;
hence they are irreducible by \Cref{irredCriterion} since $R/J_1, R/J_2$ are Gorenstein. 
Thus $I = J_1 \cap J_2$ is an irredundant irreducible decomposition of $I$ 
(note that $I$ is graded, but $J_1$ and $J_2$ are 
not graded), so $r_R(I) = 2$. By \Cref{equivTheorem}(2), $r^g_R(I) = 2$ as well:
indeed, $I = (x+y, y^3) \cap (x^2, y)$ is an irredundant graded-irreducible 
decomposition of $I$. In analogy with the monomial case, it is 
interesting to ask to what extent these graded-irreducible ideals are unique; 
this is addressed in \Cref{remark10}.\\

To see that $J_1 \cap J_2 = I$, observe that $R/I$ is an Artinian local
ring with socle $\operatorname{soc}(R/I) = (x+y, x^2)$.
Finally, if $(R, \m)$ is any local ring and $a \in 0 :_R \m$, $b \in R$ 
with $(a) \not \subseteq (b)$, then $(a) \cap (b) = 0$: if
$0 \ne ar \in (b)$ for some $r \in R$, then $r \not \in \m$ implies  
$a \in (b)$. In particular, if 
$0 \ne a, b \in 0 :_R \m$, then $(a) \cap (b) \ne 0$ iff $(a) = (b)$.
\end{example}

The ideal $I$ in \Cref{example} has some 
interesting properties, which we summarize in the following remark:

\begin{remark}\label{remark10} \
\begin{enumerate}
\item 
$R$ and $I$ are minimal in the following ways:
\begin{enumerate}
\item $R/I$ has minimal length among all graded $k$-algebras 
of finite length that are not monomial.
\item $\dim R$ and $\mu(I)$ are minimal among all 
polynomial rings $S$ and graded primary $S$-ideals $J$ such that
$J$ is reducible and is not contained in a principal ideal.
\end{enumerate}

\item The component $(x+y, y^3)$ in an irredundant decomposition of $I$ into graded-irreducibles is unique. 
Furthermore, the other component must be of the form $(x-by, y^2)$ or $(y,x^2)$, where $b \in k \setminus \{ \text{-}1 \}$, and any one of these along with $(x+y, y^3)$ forms an irredundant decomposition of $I$. 

\end{enumerate}
\end{remark}
\begin{proof}
(1): For $(a)$: Let $A$ be a graded $k$-algebra of finite length. If the length of $A$ is 1 or 2, then $A$ is isomorphic to $k$ or $k[x]/(x^2)$, respectively. If the length of $A$ is $3$, then $A$ is isomorphic to either $k[x,y]/(x^2,xy,y^2)$ or $k[x]/(x^3)$, both of which are monomial $k$-algebras. For $(b)$: If $\dim S = 1$, then every graded $S$-ideal is irreducible. 
On the other hand, if $\mu(J) \le 2$, then $J$ is either a complete intersection (hence irreducible by \Cref{irredCriterion}) or contained in a principal ideal. \\

(2): Let $K_1,K_2$ be graded-irreducible ideals such that $I = K_1 \cap K_2$ (recall that $r^g(I) = 2$). We first show that for $i = 1,2$, each $K_i$ contains a form of degree $1$. 
Suppose not. Then since $x^2+ xy, x^2-y^2$ have the least degrees among a minimal generating set for $I$, $K_i$ must contain both of them.
Observe that $K_i$ is a complete intersection ideal. Therefore, $K_i = (x^2 + xy, x^2-y^2) \not \supset I$, a contradiction. \\

Next, we show that without loss of generality $K_1 = (x+y, y^3)$ and $K_2$ can be generated by forms of degree 1 and 2. Write $\m = (x,y)R$. 
Certainly, $K_i$ cannot be generated by forms of degree 1 only: if so, then since $K_i$ is graded, $K_i$ would equal $\m$. 
Suppose both $K_i$ were generated by forms of degree 1 and 2. Then $\dim_k [K_i]_2 = 3$ (where  $[-]_j$ denotes the $j^\text{th}$ graded piece) would imply $\m^2 \subseteq K_i$, but $\m^2 \not \subseteq I$, a contradiction.
Thus without loss of generality $K_1$ is generated in degree $1$ and $3$, say $K_1 = (l,g)$, where $l$ is a linear form and $g$ is a form of degree 3. Since $y^3 \in K_1$, there exist forms $a,b$ such that $y^3 = al + bg$. Then $b \in k$, and in fact $b \ne 0$ (else $l = y \in K_1$ and then $K_1 = (y, x^2)$, a contradiction), hence $l,y^3$ also generate $K_1$. In order for $K_1$ to contain $I$, $l$ must divide both $x^2 + xy$ and $x^2 - y^2$; hence $l = x+y$ is the greatest common divisor of $x^2 + xy, x^2-y^2$.
Therefore, one has $K_1 = (x+y,y^3)$.\\

Now, we show that $K_2$ can be generated by forms of degree 1 and 2. Since $x+y$ is in $K_1$, $x+y \notin K_2$. Suppose $y \in K_2$. Then $x^2 - y^2 \in I \subseteq K_2$ implies $x^2 \in K_2$, so we conclude that $K_2 = (y, x^2)$ if $y \in K_2$. 
Let $x - by$, where $b \in k \setminus \{\text{-}1 \}$ be a linear form in $K_2$. Then $I \subseteq K_2$ is equivalent to $(b^2+b)y^2, (b^2-1)y^2, y^3 \in K_2$. Since $b \neq -1$, we conclude that $K_2 = (x-by, y^2),$ where $b \neq -1$. This proves the first part of the second statement in (2). \\

It remains to show that $L_1 \cap L_2 = I$, where $L_1 = (x + y, y^3)$ and $L_2$ is either $(y, x^2)$ or $(x - by, y^2)$, $b \in k \setminus \{-1\}$. The reasoning above shows that $I \subseteq L_1 \cap L_2$, so it suffices to show that the Hilbert functions agree. This follows since $\m^3 \subseteq I$, $\m^2 \not \subseteq L_1 \cap L_2$, $[I]_2$ is a maximal (proper) subspace of $\m^2$, and $L_1 \cap L_2$ contains no forms of degree $\le 1$.
\end{proof}

\section*{Relationship between $r(I)$ and $r(I^*)$}

Thus far, we have started with a graded object, and seen that graded and ungraded notions of irreducibility agree on graded objects. We end by briefly discussing a different setting, namely starting with a non-graded object, and passing to its closest graded approximation. \\

For any submodule $N$ of a graded module $M$, let $N^*$ denote the submodule of $N$ generated by all graded elements in $N$. Now let $\p$ be a non-graded prime ideal in a Noetherian graded ring $R$. Then the ideals $\p$ and $\p^*$, although distinct, often differ only slightly (if at all) under various properties and invariants. For example, $\height \p = \height \p^* +1$, $R_\p$ is Cohen-Macaulay (resp. Gorenstein) iff $R_{\p^*}$ is, and for a finitely generated graded $R$-module $M$, $\type(M_\p) = \type(M_{\p^*})$. In this vein, it is natural to ask how $r(I)$ compares to $r(I^*)$, for a non-graded ideal $I$. We answer this in the following special case:

\begin{prop}\label{compProp}
Let $R$ be a Noetherian graded ring, $\p$ a non-graded prime ideal of $R$, and $I$ a non-graded $\p$-primary $R$-ideal. If $I/I^*$ is principal, then $r(I) = r(I^*)$. In particular, $I$ is irreducible iff $I^*$ is irreducible. 
\end{prop}

\begin{lemma} \label{indexVsType}
Let $R$ be an $^*\!$Artinian $^*\!$local ring. Then for any maximal ideal $\m$ of $R$, $r_R(0) = \type(R_\m)$. 
\end{lemma}

\begin{proof}
Let $\q$ be the unique homogeneous prime ideal of $R$, and let $\m$ be a maximal ideal of $R$. If $\q = \m$, then $R$ is Artinian local, and the result follows from \Cref{type}. Otherwise, $\q = \m^* \ne \m$ is the largest graded subideal of $\m$, so 
\[
\type(R_\m) = \type(R_\q) = r_{R_\q}(0) = r_R(0),
\]
where the equalities follow from \cite[Theorem 1.5.9]{BH}, \Cref{type}, and the proof of \Cref{localizationLemma}, respectively (notice that $R_\q$ is an Artinian local ring).
\end{proof}

\begin{proof}[Proof of \Cref{compProp}]
First, notice that since $I$ is $\p$-primary, $I^*$ is $\p^*$-primary. The hypothesis and the numbers $r(I)$, $r(I^*)$ do not change upon going modulo $I^*$ and homogeneously localizing at $\p^*$. 
Hence we may assume that $R$ is $^*\!$Artinian $^*$local with unique maximal homogeneous ideal $\p^*$,  $\p \ne \p^*$ is a maximal ideal, and $I^* = 0$. Since $\operatorname{Ass} (R) = \{ \p^* \}$ and $I \not\subseteq \p^*$, $I$ is a principal ideal generated by a nonzerodivisor, say $f$. 
Then $r_R(I) = r_R(f) = r_{R/(f)}(0) = \type(R/(f))$ is the type of the Artinian local ring $R/(f)$, and by \Cref{indexVsType}, $r_R(I^*) = r_R(0) = \type(R_\p)$. 
If $k(\p) := R_\p/\p R_\p$, then from the isomorphism (cf. \cite[Lemma 1.2.4]{BH})
$$
\Ext^1_{R_\p}(k(\p), R_\p) \cong \Hom_{R_\p}(k(\p), R_\p/(f))
$$
we conclude that
\begin{align*}
r(I^*) &= \type(R_\p) = \dim_{k(\p)} \Ext^1_{R_\p}(k(\p), R_\p)\\
 &= \dim_{k(\p)} \Hom_{R_\p}(k(\p), R_\p/(f)) = \type(R/I) = r(I). \qedhere
\end{align*}
\end{proof}

The hypothesis that $\p$ is non-graded in \Cref{compProp} is necessary (notice that $\p$ not graded implies $I$ is not graded, but not conversely). The following examples were computed with the help of Macaulay2 \cite{M2}.
\begin{example} Let $R = k[x,y,z]$, where $k$ is a field and $\deg x = \deg y = \deg z = 1$.
\begin{enumerate}
\item For ideals $I^* = (x^3-y^3,y^3-z^3,xy,xz,yz)$ and $I = I^* +  (x^2-y^3)$, we have $r(I) = 3$, $r(I^*) = 1$, and $\sqrt{I} = \sqrt{I^*} = (x,y,z)$.

\item For ideals $I^* = (z^{3},y^{3},x^{3} y^{2},x^{5} y,x^{7})$, and $I = I^* + (x^3+xy)$, we have $r(I) = 1$, $r(I^*) = 3$, and $\sqrt{I} = \sqrt{I^*} = (x,y,z)$.
\end{enumerate}
\end{example}

The next example demonstrates that $I/I^*$ being principal is stronger than the condition that $I$ can be generated by $\mu(I^*) + 1$ elements.
\begin{example}
Let $R = k[x,y, t, t^{-1}]$, where $k$ is a field, $\deg x = 0$, and $\deg y = \deg t = 1$. Let $I = (x - y, t - 1, x^2)$. Then $I^* = (x^2,y^2)$ and $r(I) = r(I^*) = 1$. However, the ideal $I/I^*$ in $R/I^*$ requires at least two generators, for instance, $x-y, t-1$. Observe that the homogeneous minimal generating set $x^2,y^2$ of $I^*$ does not lift to part of a minimal generating set of $I$, as $x-y$ is in every minimal generating set of $I$. 
\end{example}

However, even the more general condition that $I$ can be generated by $\mu(I^*)+1$ elements is not necessary for the conclusion of \Cref{compProp} to hold, even in the simplest case when $I$ is prime (so that $r(I) = r(I^*) = 1$):
\begin{example}[Moh's primes \cite{Moh}] Let $k$ be a field, $\operatorname{char} k = 0$. Fix $n \in \mathbb{N}$ odd, $m := (n+1)/2$, $l > n(n+1)m$ with $(l, m) = 1$, and consider the ring map
\[
\varphi_n : k[x,y,z] \to k[t]
\]
\[
x \mapsto t^{nm} + t^{nm + l}, \quad y \mapsto t^{(n+1)m}, \quad z \mapsto t^{(n+2)m}
\]
Then $P_n := \ker \varphi_n$ is a height 2 non-graded prime ideal in $R := k[x,y,z]$, so $P_n^*$ is a height $1$ graded prime ideal in the UFD $R$, hence is principal. However, Moh has shown that $P_n$ requires at least $n$ generators. Thus, conditions on numbers of generators of $I$ or $I^*$ are unlikely to be necessary for $r(I) = r(I^*)$.
\end{example}

In view of these examples, we pose the following question: \\

\begin{question}
Let $R$ be a Noetherian graded ring, $\p$ a non-graded prime ideal of $R$, and $I$ a non-graded $\p$-primary $R$-ideal. What are necessary conditions for $r(I) = r(I^*)$? \\
\end{question}

\noindent \textbf{Acknowledgements:}
This project started when the authors met at the AMS sectional meeting in San 
Francisco State University in October 2014 after independently reading a question on 
\url{math.stackexchange.com}, which is answered positively in \Cref{mainTheorem}. 
We thank Professor David Eisenbud and Professor Bernd Ulrich for their advice and 
encouragement and Professor William Heinzer for reading an early draft. 
We also thank the referee for helpful comments.

\end{document}